\documentclass[12pt]{article}
\usepackage{amsmath, amssymb, amsfonts}
\usepackage{amsthm}
\usepackage[utf8]{inputenc}
\usepackage{csquotes}
\usepackage{geometry}
\usepackage{picture}
\usepackage{graphicx}
\usepackage{array}
\usepackage{booktabs}
\usepackage{hyperref}
\usepackage{url}

\title{Fractal Countability as a Constructive Alternative to the Power Set of \( \mathbb{N} \): A Meta-Formal Approach to Stratified Definability}
\author{Stanislav Semenov \\
\href{mailto:stas.semenov@gmail.com}{stas.semenov@gmail.com} \\
\href{https://orcid.org/0000-0002-5891-8119}{ORCID: 0000-0002-5891-8119}}
\date{March 27, 2025}

\theoremstyle{definition}
\newtheorem{definition}{Definition}[section]
\newtheorem*{notation}{Notation}

\theoremstyle{plain}
\newtheorem{theorem}[definition]{Theorem}

\theoremstyle{remark}

\begin{document}

\maketitle

\begin{abstract}
    Classical set theory constructs the continuum via the power set \( \mathcal{P}(\mathbb{N}) \), thereby postulating an uncountable totality. However, constructive and computability-based approaches reveal that no formal system with countable syntax can generate all subsets of \( \mathbb{N} \), nor can it capture the real line in full. In this paper, we propose \emph{fractal countability}\footnote{The concept of \emph{fractal countability} was first introduced in the context of formal limitations on constructive definability in \cite{Semenov2025FractalBoundaries}, Chapter 6. The present work develops this idea into a standalone framework, aiming to reinterpret the classical power set \( \mathcal{P}(\mathbb{N}) \) through a layered constructive lens. While the original formulation served as a meta-theoretical insight into the limits of formal systems, this article treats fractal countability as a positive and structurally grounded alternative to uncountable set constructions.} as a constructive alternative to the power set. Rather than treating countability as an absolute cardinal notion, we redefine it as a stratified, process-relative closure over definable subsets, generated by a sequence of conservative extensions to a base formal system. This yields a structured, internally growing hierarchy of constructive definability that remains within the countable realm but approximates the expressive richness of the continuum. We compare fractally countable sets to classical countability and the hyperarithmetical hierarchy, and interpret the continuum not as a completed object, but as a layered definitional horizon. This framework provides a constructive reinterpretation of power set-like operations without invoking non-effective principles.
\end{abstract}

\subsection*{Mathematics Subject Classification}
03D80 (Computability and recursion), 03E10 (Ordinal and cardinal numbers), 03B70 (Logic in computer science)

\subsection*{ACM Classification}
F.4.1 Mathematical Logic, F.1.1 Models of Computation

\section*{Introduction}

The classical conception of uncountability, grounded in set-theoretic constructions such as the power set of the natural numbers \( \mathcal{P}(\mathbb{N}) \), has long stood as a cornerstone of modern mathematics. From this foundation arises the real number continuum \( \mathbb{R} \), whose cardinality is strictly greater than that of \( \mathbb{N} \) and which serves as the basis for real analysis, measure theory, and topology. However, such constructions presuppose ontological commitments---particularly, the existence of completed infinite totalities---that lie beyond what can be accessed by constructive or algorithmic means.

Constructive mathematics, in contrast, restricts itself to those mathematical objects and processes that can be explicitly defined, algorithmically generated, or at least finitely represented within formal systems. From this perspective, the full power set \( \mathcal{P}(\mathbb{N}) \) is not just unreachable---it is inexpressible in any system that limits itself to effective operations over syntactic representations. Every definable subset of \( \mathbb{N} \) must be represented in some finite, symbolic way; and since there are only countably many such representations, only countably many subsets can be constructively specified.

This perspective builds upon prior work \cite{Semenov2025FractalBoundaries}, where the concept of the \emph{fractal boundary of constructivity} was introduced as a structural limit to formal definability. This tension lies at the heart of the present work. While classical logic embraces uncountability as a fundamental feature of the mathematical universe, constructive logic reveals that any attempt to approach uncountable sets through finite definitional processes results in layered, but ultimately countable, fragments. No single constructive method---nor even an infinite sequence of definitional extensions---can exhaust the totality of \( \mathcal{P}(\mathbb{N}) \), let alone fully instantiate the continuum.

In this article, we propose an alternative framework: \emph{fractal countability}---a stratified and process-relative model of constructive definability. Rather than treating countability as a static, absolute property (i.e., cardinal equivalence to \( \mathbb{N} \)), we develop a notion of countability indexed by syntactic growth. Specifically, we define fractally countable sets as unions of definable fragments \( S_n \), each generated by a conservative extension \( \mathcal{F}_n \) of a base formal system \( \mathcal{F}_0 \). This creates a structured, indefinitely extensible hierarchy of constructive definability---one that remains within the countable domain but allows for increasingly complex subsets of \( \mathbb{N} \) to be incorporated over time.

Our aim is twofold. First, we seek to reframe the continuum problem constructively---not by denying the existence of \( \mathbb{R} \), but by providing a formal account of what parts of it can be accessed within bounded, syntax-governed processes. Second, we offer fractal countability as a constructive analogue to the power set: a layered closure over definable subsets of \( \mathbb{N} \), which respects the finitary limitations of formal systems while capturing a rich and growing internal universe.

The structure of the article is as follows. Section 1 reviews the classical use of the power set in defining the real numbers and highlights the constructive barriers to generating all subsets of \( \mathbb{N} \). Section 2 introduces the formal framework of fractal countability. Section 3 compares this framework to classical notions of countability and to the hyperarithmetical hierarchy. Section 4 examines the implications of fractal countability for the constructive interpretation of the continuum. Section 5 discusses philosophical consequences, and Section 6 outlines directions for future research.

\section{Classical Power Set and Constructive Limitations}

The classical construction of the continuum relies on the identification of real numbers with infinite binary sequences, which in turn are often treated as characteristic functions of subsets of the natural numbers \( \mathbb{N} \). Under this perspective, the real line \( \mathbb{R} \) is defined via the power set \( \mathcal{P}(\mathbb{N}) \), typically by interpreting each subset as the binary expansion of a real number in the unit interval. This association underlies the cardinality equivalence \( |\mathbb{R}| = 2^{\aleph_0} = |\mathcal{P}(\mathbb{N})| \), a fundamental identity in classical set theory.

However, this construction depends crucially on non-constructive principles. The power set axiom asserts the existence of all subsets of any given set, regardless of their definability, computability, or representability. Consequently, \( \mathcal{P}(\mathbb{N}) \) includes vastly many subsets that are not individually specifiable within any formal system with countable syntax.

From a constructive standpoint, this poses a serious challenge. Formal systems such as Peano Arithmetic, second-order arithmetic, or type theories are inherently syntactic and finitary: they operate with countable alphabets, recursively enumerable rules, and finite derivations. Therefore, any object definable within such a system must be representable as a finite symbolic expression. Since the total number of such expressions is countable, only countably many subsets of \( \mathbb{N} \) can be defined or constructed in this way.

This limitation aligns with a well-known consequence of the Church–Turing thesis: any effectively generated function or set must be computable, and the set of computable functions is itself countable. Similarly, Gödel's incompleteness theorems demonstrate that even in powerful formal systems, there exist true statements (and implicitly, objects) that are not derivable within the system \cite{Godel1931}.

In the framework of reverse mathematics \cite{Simpson2009}, the situation becomes even clearer. The base system \( \mathrm{RCA}_0 \), which captures computable mathematics, can define only recursive subsets of \( \mathbb{N} \). Stronger systems such as \( \mathrm{ACA}_0 \) and \( \mathrm{ATR}_0 \) allow the definition of more complex sets (e.g., arithmetical or hyperarithmetical), but each of these still generates only a countable collection of subsets, bounded by the expressive power of the corresponding comprehension axioms.

In short, classical set theory assumes the totality of \( \mathcal{P}(\mathbb{N}) \), while constructive logic reveals that only a stratified fragment of this set can be reached via definable or computable means. This disjunction between assumed totality and definable fragment leads us to search for an alternative concept — one that acknowledges the layered and syntax-bound nature of constructive methods, while still capturing the richness of an internally expanding universe of subsets. The next section formalizes this idea as \emph{fractal countability}.

\section{Formal Framework of Fractal Countability}

To constructively approximate the expressive richness of \( \mathcal{P}(\mathbb{N}) \) without invoking full power set or comprehension axioms, we introduce a stratified framework of definability grounded in formal systems. The goal is to model the growth of constructively accessible subsets of \( \mathbb{N} \) as a layered and conservative expansion of syntax, rather than as an assumed totality.

\subsection{Base Systems and Conservative Extensions}

Let \( \mathcal{F}_0 \) denote a base constructive formal system, such as \( \mathrm{RCA}_0 \) from reverse mathematics or a minimal fragment of intuitionistic type theory. This system defines a set \( S_0 \subseteq \mathcal{P}(\mathbb{N}) \) consisting of all subsets of \( \mathbb{N} \) that are expressible or computable within \( \mathcal{F}_0 \).

We then consider a sequence of formal systems
\[
\mathcal{F}_0 \subseteq \mathcal{F}_1 \subseteq \mathcal{F}_2 \subseteq \cdots
\]
where each \( \mathcal{F}_{n+1} \) is a conservative syntactic extension of \( \mathcal{F}_n \), adding new definitional tools (e.g., comprehension over more complex formulas, bar recursion, or restricted forms of choice) without altering the provability of arithmetical statements in \( \mathcal{F}_n \). Each \( \mathcal{F}_n \) defines a set of subsets \( S_n \subseteq \mathcal{P}(\mathbb{N}) \), with \( S_n \subseteq S_{n+1} \).

\subsection{Fractally Countable Sets}

This layered definitional expansion motivates the following definition:

\begin{definition}[Fractally Countable Set]
Let \( \mathcal{F}_0 \) be a base formal system, and let \( \{ \mathcal{F}_n \}_{n \in \mathbb{N}} \) be a sequence of conservative syntactic extensions. A set \( S \subseteq \mathcal{P}(\mathbb{N}) \) is said to be \emph{fractally countable} (relative to \( \mathcal{F}_0 \)) if there exists a sequence of sets \( \{ S_n \} \) such that:
\begin{enumerate}
    \item \( S_0 \subseteq S_1 \subseteq \cdots \subseteq S_n \subseteq \cdots \),
    \item Each \( S_n \) consists of subsets definable within \( \mathcal{F}_n \),
    \item \( S = \bigcup_{n=0}^\infty S_n \),
    \item For every finite \( n \), \( S \nsubseteq \mathrm{Def}(\mathcal{F}_n) \), i.e., \( S \) is not definable within any single finite-stage system,
    \item Each \( \mathcal{F}_n \) remains constructive and avoids full second-order or impredicative comprehension.
\end{enumerate}
\end{definition}

\begin{notation}
We write \( S^\infty := \bigcup_{n=0}^\infty S_n \) to denote the full fractally countable closure relative to the base system \( \mathcal{F}_0 \).
\end{notation}

This definition captures the idea of a set that is fully countable in the classical sense, but whose internal structure reflects an infinite layering of definitional stages. Each fragment \( S_n \) is generated within a formal system with explicitly described rules; the set \( S^\infty \) is their union — an object that remains countable but exhibits transfinite-like complexity within a purely syntactic regime.

\subsection{Closure and Relativity}

The notion of fractal countability is inherently relative to the base system \( \mathcal{F}_0 \). A set that is fractally countable over one system may be fully definable in a stronger system, and vice versa. This relativity reflects the process-dependent nature of constructive definability and allows us to model the gradual enrichment of the expressive universe in a structured way.

Importantly, no uniform enumeration procedure for \( S^\infty \) can exist within any finite stage \( \mathcal{F}_n \), since this would collapse the hierarchy and violate the stepwise nature of the construction. This prevents paradoxes such as Richard’s paradox and ensures that each definitional jump adds genuinely new content not available at previous levels.

The next section will compare fractally countable sets to classical notions such as the set of all countable subsets, the hyperarithmetical hierarchy, and \( \Delta^1_1 \) definability.

\section{Comparison with Classical Countability and Hyperarithmetical Sets}

Fractal countability, as introduced in the previous section, offers a layered model of definability that remains within the bounds of classical countability but extends far beyond basic computability. To understand its reach and limitations, we compare it to standard notions such as classical countability, the hyperarithmetical hierarchy, and \( \Delta^1_1 \) definability. We also explore how different choices of the base system \( \mathcal{F}_0 \) affect the resulting structure, and consider potential applications to reverse mathematics and proof theory.

\subsection{\texorpdfstring{Comparison with Classical Countability}{Comparison with Classical Countability}}

In classical set theory, a set is countable if there exists a bijection with \( \mathbb{N} \), regardless of whether such a bijection is computable or even definable. This notion admits the existence of countable sets whose elements are not describable by any effective means.

In contrast, fractally countable sets are not merely countable in cardinality — they are countable \emph{by construction}. Each element arises from a definitional process tied to a specific formal system. The set \( S^\infty \) represents the closure under all such definitional layers relative to the base system \( \mathcal{F}_0 \). Thus, fractal countability refines the classical notion by imposing a strict syntactic structure and generation history on the elements of the set.

A critical consequence of this constraint is that fractal countability \emph{excludes classically countable sets that lack constructive descriptions}. Sets whose elements are postulated via non-effective or impredicative means — such as certain subsets of \( \mathbb{N} \) that arise from the axiom of choice — have no place in this framework. Fractal countability, therefore, not only reproduces a controlled fragment of classical countability, but also acts as a filter against so-called "pathological" constructions.

\subsection{\texorpdfstring{Relation to Hyperarithmetical and \( \Delta^1_1 \) Hierarchies}{Relation to Hyperarithmetical and Delta-1-1 Hierarchies}}

The hyperarithmetical hierarchy (denoted \( HYP \)) consists of all subsets of \( \mathbb{N} \) that are computable relative to some finite iteration of the Turing jump on the empty set \cite{Odifreddi1989,Soare2016}. These correspond to sets definable in second-order arithmetic with restricted \( \Pi^1_1 \)-comprehension and form a natural boundary for many semi-constructive systems.

Let us consider the sequence \( S_n = \mathbb{N} \cap \Delta^0_n \), where each \( S_n \) consists of sets recursive in \( \emptyset^{(n)} \), the \( n \)-th Turing jump of the empty set. Then
\[
S^\infty = \bigcup_{n=0}^\infty S_n = HYP.
\]
This shows that fractal countability can reconstruct \( HYP \) under a suitable sequence of extensions \( \mathcal{F}_n \), where each \( \mathcal{F}_n \) defines the sets computable relative to \( \emptyset^{(n)} \).

However, fractal countability is not intrinsically limited to \( HYP \). Its construction does not rely on external notions such as oracle computation or ordinal notations, but proceeds through syntactic extensions of formal systems. Therefore, if the systems \( \mathcal{F}_n \) incorporate stronger constructive principles — such as bar recursion, bounded choice, or inductive definitions — the resulting set \( S^\infty \) may extend beyond the classical hyperarithmetical hierarchy, while remaining within a constructively interpretable regime.

\subsection{\texorpdfstring{Dependence on the Base System \( \mathcal{F}_0 \)}{Dependence on the Base System F₀}}

A central feature of fractal countability is its relativity to the base system \( \mathcal{F}_0 \). Different choices of \( \mathcal{F}_0 \) may lead to non-equivalent or even incomparable hierarchies of definable sets. For example, choosing \( \mathrm{RCA}_0 \) yields a hierarchy grounded in computable mathematics, while starting from a dependent type-theoretic base may yield hierarchies defined by canonical terms and constructive identity types.

This raises an important issue: \emph{how should one choose the base system}? The answer depends on the epistemological context and the application domain. If one aims to study the limits of algorithmic definability, \( \mathrm{RCA}_0 \) or similar systems are natural choices. If one is interested in internal reasoning in constructive proof assistants, a type-theoretic foundation may be more appropriate.

Crucially, hierarchies based on different \( \mathcal{F}_0 \) may define distinct families of sets that are not mutually embeddable. For instance, certain sets definable in a type-theoretic hierarchy may not be definable in any arithmetic fragment, and vice versa. This suggests the need for a general theory of \emph{hierarchy comparison}, capable of measuring and relating definitional power across foundational systems.

\subsection{Potential Applications and Open Directions}

Fractal countability may serve as a useful tool in both reverse mathematics and proof theory. In reverse mathematics, one could analyze which comprehension axioms are needed to reach a given stage \( S_n \), or to reconstruct the full closure \( S^\infty \) from a base \( \mathcal{F}_0 \) \cite{FriedmanSimpson1978}. This could lead to a new spectrum of subsystems, defined not just by logical axioms, but by their position in a definitional hierarchy.

In proof theory, the syntactic layering of \( \mathcal{F}_n \) offers a constructive analogue to ordinal analysis. Instead of indexing systems by ordinals, one tracks their enrichment through definitional closure. This approach may clarify the growth of provable total functions, and help calibrate the expressive strength of formal systems in a finer-grained, syntactic manner.

Practically, one could also attempt to verify that particular sets, such as those arising in effective versions of classical theorems, belong to a specific \( S_n \) — thereby giving \emph{constructive witnesses} to their definability. For example, proving that the set of solutions to an arithmetically parameterized compactness condition lies in \( S_2 \) would tie a classical result to a concrete constructive stage.

\paragraph{Open Questions.} Several questions remain open:
\begin{itemize}
    \item How does \( S^\infty \) compare precisely to \( HYP \), \( \Delta^1_1 \), and the analytic hierarchy? Are there cases where fractal countability is strictly stronger or weaker?
    \item What principles (e.g., bar recursion, restricted choice, definitional reflection) are admissible in \( \mathcal{F}_n \) without compromising constructivity?
    \item Are there ``natural'' sequences of extensions \( \mathcal{F}_n \) corresponding to common hierarchies in logic, such as those in reverse mathematics or proof assistants?
    \item Can we develop a general framework for comparing definability hierarchies that originate from distinct base systems?
\end{itemize}

These questions point to a broader program: understanding definability not as a static global property, but as a layered and system-relative process — one that can be analyzed, compared, and applied across both constructive and classical contexts.

\section{The Continuum as a Constructive Horizon}

In classical set theory, the continuum is defined via the power set \( \mathcal{P}(\mathbb{N}) \), postulating the existence of all subsets of natural numbers. This construction yields a set of cardinality \( 2^{\aleph_0} \), forming the basis for real analysis and classical models of computation over the reals.

However, from a constructive perspective, such completeness is inaccessible. As shown in the previous sections, no countable formal system can generate an uncountable collection of definable sets — even when extended iteratively through syntactic layers. This leads us to reinterpret the continuum not as a completed object, but as a process-relative \emph{horizon} of definability: a boundary that can be approached by increasingly enriched systems, but never fully reached.

\subsection{Formal Inaccessibility and the Role of Fractal Countability}

Let \( \mathcal{F}_0 \) be a countable base formal system, and let \( S^\infty \) be the fractally countable union over all definable stages \( S_n \) generated by extensions \( \mathcal{F}_n \). Then:

\begin{theorem}[Fractal Inaccessibility of the Continuum]
Let \( \mathcal{F}_0 \) be any countable formal system, and \( \{ \mathcal{F}_n \} \) a sequence of constructive conservative extensions. Then the fractally countable closure \( S^\infty \subseteq \mathcal{P}(\mathbb{N}) \) remains countable, and cannot coincide with the full power set \( \mathcal{P}(\mathbb{N}) \) unless non-constructive comprehension principles are introduced.
\end{theorem}

\begin{proof}[Sketch]
By construction, each \( \mathcal{F}_n \) produces only countably many definable subsets of \( \mathbb{N} \). The countable union of countable definable sets is countable. Since \( \mathcal{P}(\mathbb{N}) \) is uncountable in classical cardinality, \( S^\infty \) is strictly smaller unless uncountable sets are assumed a priori.
\end{proof}

Thus, fractal countability offers a way to interpret the continuum not as a totality but as a definitional horizon — a layered closure that never captures the full classical object but yields increasingly expressive approximations within constructive bounds.

\subsection{From Object to Process: A Shift in Ontology}

This view aligns with intuitionistic and constructivist interpretations, where the continuum is not a set of completed entities but an unfolding process. Brouwer's notion of the continuum as a "medium of free becoming" resonates with this \cite{Brouwer1927}: real numbers are not discovered as completed values, but constructed through successively refined approximations.

Fractal countability gives this philosophical view a formal counterpart. The continuum is not a set but a stratified structure of definability, with each layer reflecting an extension of our formal language. This reframes uncountability: not as a question of size, but of \emph{expressive inaccessibility}.

\subsection{Beyond Representations: Constructive Internalization}

Unlike classical codings — which embed \( \mathbb{R} \) into sequences or trees — fractal countability internalizes the expansion process. The goal is not to simulate the continuum syntactically, but to analyze how far formal systems can reach in approximating it from within.

This has implications for how we understand the real line in constructive mathematics:

\begin{itemize}
    \item The continuum becomes system-relative: \( S^\infty \) depends on the choice of \( \mathcal{F}_0 \).
    \item Continuity becomes layered: each \( S_n \) adds expressivity but cannot collapse the hierarchy.
    \item Completeness is reinterpreted: rather than being a global assumption, it becomes a structural property of the extension process.
\end{itemize}

\subsection{Toward a New Constructive Foundation}

Fractal countability does not deny the utility of the continuum in classical mathematics. Rather, it proposes a refined approach: one where definability is not global but stratified, where power-set constructions are replaced by process-relative closures, and where mathematical objects grow with the expressive power of the systems in which they live.

In this view, the real line is not an already-formed set, but the evolving boundary of what can be said about it — an expanding but always countable trace of constructive thought.

This perspective opens the way to a foundational approach where:

\begin{itemize}
    \item Constructive systems are calibrated by their ability to approximate \( \mathbb{R} \) via \( S_n \);
    \item The expressive ceiling is visible and formally bounded;
    \item Classical structures emerge not by fiat but through stratified construction.
\end{itemize}

In the next section, we turn to the philosophical implications of this perspective, situating fractal countability within the broader landscape of realism, constructivism, and foundational pluralism.

\subsection{Examples of Stratified Continuum Approximations}

To better illustrate the constructive and hierarchical nature of fractal countability, we now present two concrete examples of how sequences \( \{ \mathcal{F}_n \} \) can be constructed to progressively approximate the continuum. These examples emphasize the role of syntactic extension in defining richer classes of subsets of \( \mathbb{N} \), while remaining within a countable framework.

\paragraph{Example 1: From Computable to Arithmetical Sets.}
Let the base system be \( \mathcal{F}_0 = \mathrm{RCA}_0 \), which captures recursive comprehension and defines only computable sets. We take \( \mathcal{F}_1 = \mathrm{ACA}_0 \), extending comprehension to all arithmetical formulas.

Then:
\begin{itemize}
    \item \( S_0 \) consists of all computable subsets of \( \mathbb{N} \).
    \item \( S_1 \) consists of all subsets of \( \mathbb{N} \) definable by arithmetical formulas, i.e., recursive in \( \emptyset' \).
\end{itemize}

The resulting closure \( S^\infty = S_1 \) if the chain stops, or it continues to grow if we extend further (e.g., \( \mathcal{F}_2 = \mathrm{ATR}_0 \), etc.). This matches the stratified development in reverse mathematics \cite{FriedmanSimpson1978}, with each step enriching definability while preserving countability. See Table~\ref{tab:fractal-approximation} for an overview of this approximation hierarchy.

\begin{table}[ht]
\centering
\renewcommand{\arraystretch}{1.4}
\resizebox{\textwidth}{!}{%
\begin{tabular}{@{} c c c p{8.2cm} @{}}
\toprule
\textbf{Level \(n\)} & \textbf{System \( \mathcal{F}_n \)} & \textbf{Set \( S_n \)} & \textbf{Description} \\
\midrule
0 & \( \mathrm{RCA}_0 \) & Computable sets & Recursive (decidable) subsets of \( \mathbb{N} \); basic computable mathematics. \\
1 & \( \mathrm{ACA}_0 \) & Arithmetical sets & Subsets definable by arithmetical formulas; recursive in the Turing jump \( \emptyset' \). \\
2 & \( \mathrm{ATR}_0 \) & Hyperarithmetical sets & Closure under arithmetical transfinite recursion; encodes well-founded inductive definitions. \\
3 & \( \mathcal{F}_3 \) (e.g., with bar recursion) & Beyond \( HYP \)? & Sets definable via higher-type recursion; may include functionals or trees beyond the hyperarithmetical. \\
\bottomrule
\end{tabular}%
}
\caption{Stratified approximation to \( \mathcal{P}(\mathbb{N}) \) via fractal countability}
\label{tab:fractal-approximation}
\end{table}

Figure~\ref{fig:stratified-example} shows a typical progression of \( \mathcal{F}_n \) from recursive to hyperarithmetical sets, corresponding to increasing comprehension axioms.

\begin{figure}[h]
\centering
\begin{tabular}{@{}ll@{}}
\( \mathcal{F}_0 = \mathrm{RCA}_0 \): & \( S_0 = \) computable sets \\
\( \mathcal{F}_1 = \mathrm{ACA}_0 \): & \( S_1 = \) arithmetical sets \\
\( \mathcal{F}_2 = \mathrm{ATR}_0 \): & \( S_2 = \) hyperarithmetical sets \\
\end{tabular}
\caption{Example of stratified continuum approximation under classical subsystems}
\label{fig:stratified-example}
\end{figure}

\paragraph{Example 2: Constructive Extensions via Bar Recursion.}
Let the base be Heyting arithmetic \( \mathcal{F}_0 = \mathrm{HA} \), and let \( \mathcal{F}_1 \) extend \( \mathrm{HA} \) with bar recursion of finite type up to level 1 — a principle often admitted in constructive type theory and systems such as Gödel's System T or proof assistants like Agda \cite{Troelstra1988}.

Then:
\begin{itemize}
    \item \( S_0 \) includes only primitive recursive and constructively provable sets.
    \item \( S_1 \) includes functionals definable with higher-type recursion — for example, solutions to totality conditions over trees or dependent choices over finite data.
\end{itemize}

In this setting, \( S^\infty \) may include sets not accessible through arithmetical comprehension alone. The definitional closure could even reach beyond the hyperarithmetical hierarchy, depending on the richness of the recursion principles used. See Table~\ref{tab:constructive-fractal-approximation} for a parallel stratification over constructive and type-theoretic foundations.

These examples demonstrate how fractal countability supports structurally enriched yet still countable approximations to the continuum. Each choice of \( \mathcal{F}_n \) encodes not just a logical strength, but a constructive epistemology — revealing how mathematics unfolds relative to its formal resources.

\begin{table}[ht]
\centering
\renewcommand{\arraystretch}{1.4}
\resizebox{\textwidth}{!}{%
\begin{tabular}{@{} c c c p{8.2cm} @{}}
\toprule
\textbf{Level \(n\)} & \textbf{System \( \mathcal{F}_n \)} & \textbf{Set \( S_n \)} & \textbf{Description} \\
\midrule
0 & \( \mathrm{HA} \) & Primitive recursive sets & Constructively provable sets under intuitionistic logic; limited to basic recursion. \\
1 & \( \mathrm{HA} + \text{Bar Recursion} \) & Bar-recursively definable sets & Sets defined via controlled dependent choice and tree-based recursion \cite{Troelstra1988}. \\
2 & \( \text{System } T \) & Functionals of finite type & Contains higher-type functionals; definable total functions over all finite types. \\
3 & \( \mathrm{MLTT} + \text{W-types} \) & Inductive families & Sets defined by dependent inductive types; includes structured infinite data. \\
\bottomrule
\end{tabular}%
}
\caption{Fractal definability hierarchy over constructive and type-theoretic systems}
\label{tab:constructive-fractal-approximation}
\end{table}

\section{Philosophical Implications}

The shift from classical to constructive views on the continuum marks not only a technical divergence but also a fundamental ontological and epistemological transformation. In this section, we reflect on how the notion of fractal countability realigns foundational perspectives on mathematical infinity, objecthood, and definability.

\subsection{From Completed Totalities to Constructive Horizons}

In classical set theory, the continuum \( \mathbb{R} \) is treated as a completed uncountable totality, grounded in the power set \( \mathcal{P}(\mathbb{N}) \). This view presupposes the existence of a vast, well-defined domain of subsets, most of which are not individually describable. The classical universe is closed under arbitrary comprehension, embracing the full consequences of the power set axiom.

By contrast, the notion of fractal countability replaces this static conception with a process-relative perspective. The continuum becomes a horizon: a constructively unreachable boundary toward which formal systems may progress via definitional extensions, but which can never be fully crossed. This mirrors Brouwer's vision of the continuum as a \emph{medium of free becoming} \cite{Brouwer1927} and gives it a precise meta-formal formulation.

\subsection{Stratification and Relative Ontology}

In the fractal framework, mathematical existence becomes stratified. There is no absolute domain of sets, but rather a hierarchy of definable stages \( S_n \), each indexed by a conservative extension \( \mathcal{F}_n \). This leads to a system-relative ontology\footnote{Unlike in ZFC, where \(\mathbb{R}\) is unique up to bijection, here \(\mathbb{R}_{\mathcal{F}_0}\) is determined by the base system and may vary accordingly.}: what exists depends on the expressive power of the system from which it is viewed. The totality of real numbers is no longer an object but an evolving structure of potential definability. This stratified reinterpretation of the continuum is visualized schematically in Figure~\ref{fig:fractal-continuum-diagram}.

\begin{figure}[ht]
\centering
\setlength{\unitlength}{1cm}
\begin{picture}(12,5.2)

% Classical view
\put(0.5,4.0){\framebox(11,1){\textbf{ZFC View:} \(\mathbb{R} = \mathcal{P}(\mathbb{N})\) as completed actual infinity}}
\put(6,3.9){\vector(0,-1){0.5}}

% Fractal approach
\put(0.5,2.0){\framebox(11,1.2){\textbf{Fractal View:} Stratified definability: \( \mathbb{R}_{\mathcal{F}_0} = \bigcup S_n \)}}
\put(6,1.9){\vector(0,-1){0.5}}

% Constructive limit
\put(0.5,0.2){\framebox(11,1){\textbf{Constructivist Limit:} No system \( \mathcal{F}_n \) can fully reach \(\mathbb{R}\)}}

\end{picture}
\caption{Philosophical trajectories of the continuum: from classical totality to stratified definability}
\label{fig:fractal-continuum-diagram}
\end{figure}
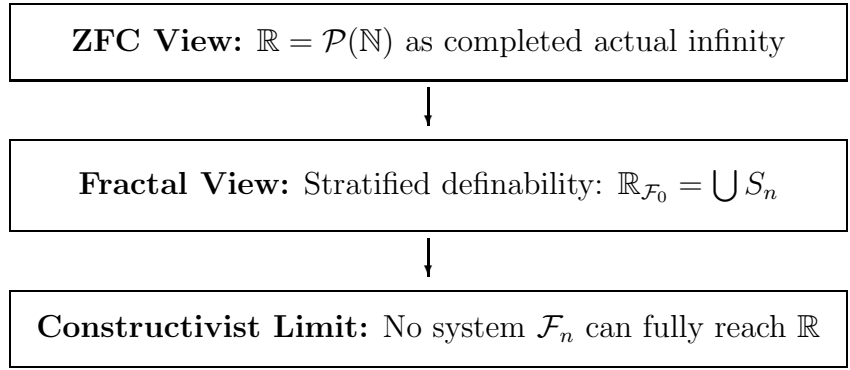

This perspective aligns with Feferman's notion of \emph{predicative mathematics}, where universes are constructed step-by-step in a stratified, ordinal-indexed fashion \cite{Feferman2005}. Our approach similarly builds the definable universe layer by layer, not by appeal to impredicative comprehension, but through cumulative definitional closure. It emphasizes the meta-mathematical visibility of each stage and the transparency of constructive growth.

\subsection{Bridging Formalism and Constructivism}

Fractal countability also mediates between formalist and constructivist traditions. From the formalist point of view, mathematics is the manipulation of symbols under rules. From the constructivist view, mathematical entities must be generated by effective means. The stratified definability framework honors both: it operates within formal systems but insists on constructible semantics.

Rather than reject the continuum as a classical fiction, we reinterpret it as a structured idealization: a formal projection of what could be constructed in principle. The continuum becomes a limit of definability progression\footnote{In computational settings, this limit may be approximated via proof assistant universes or stratified type hierarchies.}, not a completed set but a direction of formal unfolding. This brings clarity to debates about realism and anti-realism in mathematics, showing how nonconstructive ideals can be reconciled with system-relative epistemology.

\subsection{Implications for Foundational Pluralism}

Finally, the fractal approach supports foundational pluralism. Different base systems \( \mathcal{F}_0 \) yield different closure chains \( S_n \), leading to distinct yet internally coherent hierarchies. There is no privileged formal universe, only overlapping trajectories of expressibility.

This plurality is not a weakness but a resource: it allows for tailored formal systems adapted to specific epistemic goals. For example, a proof assistant may implement a version of \( \mathcal{F}_0 \) compatible with constructive type theory, while a classical analyst may prefer \( \mathrm{ACA}_0 \). Each yields its own notion of continuum approximation, valid within its own framework.

\subsection{Objections and Replies}

\paragraph{Classical Challenges.} A standard objection is that rejecting \( \mathcal{P}(\mathbb{N}) \) as a totality risks losing key results from analysis, such as the full definition of the Lebesgue measure on Borel sets. While true in the classical framework, fractal countability does not deny the value of such tools — it situates them as idealizations that may be approximated constructively in layered form. Many analytic structures, such as effectively compact spaces or constructively integrable functions, remain meaningful within specific \( \mathcal{F}_n \).

\paragraph{Constructivist Radicalism.} From a stricter constructivist standpoint, one might ask: why retain the language of "the continuum" at all, if it remains an inaccessible horizon? The answer is pragmatic and structural. Fractal countability acknowledges the classical ideal, but reframes it as a regulative structure for constructive expansion — much like how constructive analysis retains notions of limits, but reinterprets them in terms of verifiable convergence.

\subsection{Connections to Contemporary Metamathematics}

The stratified approach to definability resonates with recent developments in type theory and constructive foundations. In particular, Univalent Foundations and Homotopy Type Theory (HoTT) construct universes in stages, stratified by type levels and homotopical structure. These hierarchies parallel the \( S_n \) sequence, albeit with richer categorical semantics.

Moreover, in the spirit of Harvey Friedman's finitistic reductions, one might ask whether specific fragments of analysis can be reconstructed within a bounded \( S_n \). The fractal framework allows for such bounded representations, providing a middle ground between proof-theoretic conservativity and constructive executability.

\section{Future Directions and Conclusion}

The framework of fractal countability opens several lines of inquiry across logic, computability, and foundations of mathematics:

\begin{itemize}
  \item \textbf{Stratified complexity analysis:} Can one classify mathematical theorems based on the minimal level \( S_n \) in which their objects or functions appear? This would provide a fine-grained alternative to classical reverse mathematics.

  \item \textbf{Cross-foundational hierarchy comparisons:} Given two base systems \( \mathcal{F}_0 \) and \( \mathcal{F}'_0 \), can we meaningfully compare their resulting hierarchies of \( S_n \)? What transformations or embeddings preserve expressivity?

  \item \textbf{Constructive extensions beyond HYP:} Can fractal countability reach definable reals outside the hyperarithmetical hierarchy using extensions such as bar recursion, choice principles, or type universes?

  \item \textbf{Applications in formal systems:} How can stratified definability be reflected in proof assistants and type-theoretic languages such as Agda or Lean, where definability and totality are syntactically enforced?

  \item \textbf{Epistemic modeling:} Does a stratified view of definability offer better alignment with how knowledge and formalization develop in practice — for example, in scientific theories or constructive physics?
  
  \item \textbf{Applications in cryptography:} Can layered definability model access-control hierarchies or formalize stratified secrecy in provably secure systems?
\end{itemize}

Table~\ref{tab:research-directions} summarizes some of the main research avenues suggested by the stratified framework, highlighting how technical developments map onto philosophical questions.

\begin{table}[ht]
\centering
\renewcommand{\arraystretch}{1.4}
\resizebox{\textwidth}{!}{%
\begin{tabular}{@{} p{4.2cm} c p{8.2cm} @{}}
\toprule
\textbf{Direction} & \textbf{Type} & \textbf{Description} \\
\midrule
Stratified complexity & Technical & Classify theorems by the minimal \( S_n \) needed to express their content; provides a new lens on mathematical depth via definability. \\
\addlinespace
Hierarchy comparison & Meta-theoretical & Study translations between hierarchies induced by different base systems \( \mathcal{F}_0 \); supports pluralism of foundations. \\
\addlinespace
Beyond HYP & Logical & Explore whether extensions like bar recursion or dependent choice allow constructing reals outside the hyperarithmetical hierarchy. \\
\addlinespace
Integration with proof assistants & Applied & Incorporate stratified definability levels \( S_n \) into formal tools like Agda or Lean; calibrate expressive power against logical layers. \\
\addlinespace
Epistemic modeling & Philosophical & Use stratified closure to model knowledge growth, scientific theory evolution, and process-relative objectivity. \\
\addlinespace
Stratified secrecy & Security & Model multi-level access and verifiable information disclosure using definability layers; potential applications in formal cryptographic frameworks. \\
\bottomrule
\end{tabular}%
}
\caption{Stratified research directions connecting logic, foundations, and definability}
\label{tab:research-directions}
\end{table}

\paragraph{Conclusion.} Fractal countability offers a constructive, stratified alternative to the power set of \( \mathbb{N} \), preserving definability while avoiding the ontological leap of classical uncountability. It reframes the continuum not as a fixed set but as a growing horizon of expressibility. Through its emphasis on layered closure, meta-theoretic transparency, and system-relative semantics, it provides a foundation for pluralistic and computationally grounded mathematics. Further exploration may reveal its role as a unifying framework between constructive mathematics, reverse analysis, and proof-theoretic metascience.

% \bibliographystyle{plain}
% \bibliography{references}

\end{document}